\documentclass[11pt, reqno]{amsart}
\usepackage{a4,amsmath,amssymb,amscd,verbatim,bbm,graphicx,enumerate, blkarray}

\usepackage{amsfonts,amsthm}
\usepackage[colorlinks,citecolor=cyan,linkcolor=magenta]{hyperref}
\usepackage{graphicx}
\usepackage{bmpsize}
\usepackage{MnSymbol}
\usepackage{color}
\usepackage{mathrsfs}

\usepackage[nameinlink]{cleveref}

\newtheorem{theorem}{Theorem}
\newtheorem{proposition}[theorem]{Proposition}
\newtheorem{corollary}[theorem]{Corollary}
\newtheorem{lemma}[theorem]{Lemma}

\theoremstyle{definition}

\newtheorem{definition}[theorem]{Definition}

\newtheorem{remark}[theorem]{Remark}

\numberwithin{equation}{section}
\numberwithin{theorem}{section}

\newcommand{\G}{\Gamma}

\renewcommand{\l}{\lambda}

\renewcommand{\r}{\rho}

\newcommand{\Dc}{{\mathcal D}}

\newcommand{\C}{{\mathbb C}}

\newcommand{\R}{{\mathbb R}}

\newcommand{\calB}{\mathcal{B}}

\newcommand{\calD}{\mathcal{D}}

\newcommand{\calF}{\mathcal{F}}

\newcommand{\calT}{\mathcal{T}}

\newcommand{\scrD}{\mathscr{D}}

\newcommand{\Hy}{\mathbb{H}}

\newcommand{\al}{\alpha}

\newcommand{\Gam}{\Gamma}
\newcommand{\del}{\delta}

\newcommand{\Del}{\Delta}
\newcommand{\ep}{\epsilon}

\newcommand{\lam}{\lambda}
\newcommand{\Lam}{\Lambda}
\newcommand{\sig}{\sigma}

\newcommand\PSL{\operatorname{PSL}}

\newcommand\Out{\textnormal{Out}}

\newcommand{\ra}{\rightarrow}

\def\({\left(}
\def\){\right)}

\def\l\{{\left\{}
\def\r\}{\right\}}
\def\wt{\widetilde}
\def\wh{\widehat}
\def\gm{\mathfrak{g}}
\def\fraka{\mathfrak{a}}
\def\frakD{\mathfrak{D}}
\def\frakR{\mathfrak{R}}
\def\wbar{\overline}
\def\ov{\overline}
\def\id{{\rm id}}

\def\Cay{{\rm Cay}}
\def\Leb{{\rm Leb}}
\def\ev{{\rm ev\,}}
\def\1{{\bf 1}}

\def\gr{{\rm gr}}
\def\conj{{\bf conj}}

\def\ac{{\rm ac}}
\def\sing{{\rm sing}}

\def\H{{\mathbb H}}

\def\gf{{\frak g}}
\def\bS{{\bf S}}

\def\CAT{{\rm CAT}}

\def\Dil{{\rm Dil}}

\DeclareMathOperator{\diam}{diam}

\newcommand\numberthis{\addtocounter{equation}{1}\tag{\theequation}}

\newtheorem*{remark*}{Remark}

\title[Approximate marked length spectrum rigidity in coarse geometry]{Approximate marked length spectrum rigidity in coarse geometry}

\author{\small{Stephen Cantrell and Eduardo Reyes}}

\begin{document}

\begin{abstract}
We compare the marked length spectra of isometric actions of groups with non-positively curved features. Inspired by the recent works of Butt we study approximate versions of marked length spectrum rigidity. 
 We show that for pairs of metrics, the supremum of the quotient of their marked length spectra is approximately determined by their marked length spectra restricted to an appropriate finite set of conjugacy classes.
Applying this to fundamental groups of closed negatively curved Riemannian manifolds allows us to refine Butt's result. Our results however apply in greater generality and do not require the acting group to be hyperbolic. For example we are able to compare the marked length spectra associated to mapping class groups acting on their Cayley graphs or on the curve graph.
\end{abstract}
\maketitle

\section{Introduction}\label{sec:intro}

\subsection{Riemannian manifolds}
Suppose that $(M,\mathfrak{g})$ is a closed, negatively curved Riemannian manifold. There is a bijection between the (non-trivial) conjugacy classes $\conj=\conj(\G)$ in the fundamental group $\G := \pi_1(M)$ and the oriented closed geodesics on $(M,\mathfrak{g})$. Through this correspondence we obtain a map
\[
\ell_\mathfrak{g} : \conj(\G) \to \R_{>0}
\]
that sends each conjugacy class to the length of the corresponding closed geodesic. This map, which is referred to as the \textit{marked length spectrum} of $(M,\mathfrak{g})$, records the lengths of the closed geodesics on $(M,\mathfrak{g})$ and orders them according to the conjugacy classes in $\G$. The marked length spectrum rigidity conjecture (attributed to Burns and Katok \cite{burns.katok}) predicts that $\ell_\mathfrak{g}$ determines the isometry class of $\mathfrak{g}$ amongst all negatively curved metrics on $M$, i.e. if $\mathfrak{g}_\ast$ is another negatively curved Riemannian metric on $M$ with the same marked length spectrum as $\mathfrak{g}$ then $\mathfrak{g}$ and $\mathfrak{g}_\ast$ must be isometric. The conjecture has been confirmed independently by Croke \cite{croke} and Otal \cite{otal} for surfaces. In higher dimensions, the conjecture is known to hold if one of the metrics is locally symmetric \cite{BCG,hamenstadt}.  A recent breakthrough of Guillarmou and Lefeuvre \cite{guilef} verifies a local version of the conjecture for Riemannian manifolds with Anosov geodesic flow and non-positive curvature. In general the conjecture remains open.

When $S$ is a closed, orientable surface of genus $g \ge 2$, the marked hyperbolic (i.e. constant curvature $-1$) metrics on $S$ form a moduli space called Teichm\"uller space $\calT(S)$. It is well-known, thanks to Frenchel-Nielsen coordinates, that the length spectra of points in $\calT(S)$ are determined by their values taken on finitely many conjugacy classes. In fact there are $6g - 6$ conjugacy classes that determine the full marked length spectrum (and hence isometry class) of hyperbolic metrics on $S$ \cite{schmutz}. 
This property does not hold for arbitrary metrics of (variable) negative curvature. Indeed, if $\mathfrak{g}$ is a negatively curved metric on $S$ and $F$ is any finite collection of closed geodesics on $(S,\mathfrak{g})$, then we can perturb $\mathfrak{g}$ on the complement of $F$ to obtain a new metric $\mathfrak{g}'$. We can do this so that $\mathfrak{g}'$ non-isometric to $\mathfrak{g}$, is still negatively curved and so that $\ell_{\mathfrak{g}}$ and $\ell_{\mathfrak{g}'}$ coincide on $F$.

It is then natural to ask to what extent the marked length spectrum of a negatively curved Riemannian metric is determined by its marked length spectrum restricted to a finite collection of closed geodesics. This question was considered by Butt in \cite{butt.2} and \cite{butt.1} in which she proves an \textit{approximate length spectrum rigidity} result. Suppose $\mathfrak{g}, \mathfrak{g}_\ast$ are negatively curved Riemannian  metrics on a closed manifold $M$. Roughly speaking Butt's result \cite[Thm.~1.2]{butt.2} states that there exist constants $C>0, \alpha \in (0,1)$ (depending on geometric data) such if $|\ell_\mathfrak{g}[g]/\ell_{\mathfrak{g}_\ast}[g] - 1| < \epsilon$ for all $[g] \in \conj(\G)$  with $\ell_\mathfrak{g}[g] < L$ then $|\ell_\mathfrak{g}[g]/\ell_{\mathfrak{g}_\ast}[g] - 1| < \epsilon + CL^{-\alpha}$ for \emph{all} conjugacy classes $[g] \in \conj(\G)$. That is, for a pair of negatively curved metrics $\mathfrak{g}, \mathfrak{g}_\ast$, precise global bounds comparing their marked length spectra can be deduced from bounds taken over \emph{finitely many} conjugacy classes. The first aim of this paper is to prove a refinement of Butt's results through a concise, coarse geometric argument. We show the following in which we write $\conj'=\conj'(\G)$ for the non-torsion conjugacy classes in $\G$. 
 
\begin{theorem}\label{thm.butt}
Let $(M,\gm)$ and $(M_\ast,\gm_\ast)$ be closed Riemannian manifolds with sectional curvatures lying in the interval $[-\Lam^2,-\lam^2]\subset \R_{<0}$, and let $i_\gm$ be the injectivity radius of $(M,\gm)$. Let $\G$ be the fundamental group of $M$ and suppose there exists an isomorphism $f_\ast:\G \ra \pi_1(M_\ast)$. Then for any $0<\ep_0<1$ there exists some $L_0=L_0(\G,\lam,\Lam,i_\gm,\ep_0)$ such that the following holds.  If $L\geq L_0$, $0<\ep\leq \ep_0$ and 
\begin{equation}\label{eq.buttassumption}
1-\ep\leq \frac{\ell_{\gm_\ast}[f_\ast g]}{\ell_{\gm}[g]}\leq 1+\ep  \ \ \text{ for all } [g]\in \conj' \text{ with }\ell_{\gm_\ast}[f_\ast g]\leq L,
\end{equation}
then 
\begin{equation}\label{eq.buttconclusion}
1-\left(\ep+\frac{C}{L}\right)\leq \frac{\ell_{\gm_\ast}[f_\ast g]}{\ell_\gm[g]}\leq 1+\left(\ep+\frac{C}{L}\right)
\end{equation}
for all $[g]\in \conj'$,
where $C$ depends only on $\G, \lam,\Lam, i_\gm$ and $\ep_0$.
\end{theorem}
This theorem provides a stronger global bound to Butt's result as the error in the comparison inequality (\ref{eq.buttconclusion}) decays like $L^{-1}$ opposed to $L^{-\alpha}$ for $\alpha \in (0,1)$.


\subsection{General actions on hyperbolic spaces}\label{sec.introhyperbolic}
To prove Theorem \ref{thm.butt} we use ideas and techniques from geometric group theory. This perspective allows us to prove much more general results. Suppose that a group $\G$ admits an isometric action on a possibly asymmetric metric space $X=(X,d_X)$ (so $d_X$ satisfies all the actions of a metric, except for possibly symmetry). Then the \emph{(stable) translation length function} associated to this action is the map
\[
\ell_d : \conj(\G) \to \R_{\ge 0} \ \text{ given by } \ \ell_X[g] = \lim_{k\to\infty} \frac{d_X(x,g^k\cdot x)}{k}
\]
where $x \in X$ is a fixed base point. When $(M,\mathfrak{g})$ is a closed, negatively curved Riemannian manifold we can recover the marked length spectrum of $\mathfrak{g}$ as the stable translation length function associated to the action of $\pi_1(M)$ on the universal cover $(\widetilde{M},\widetilde{g})$ of $(M,\mathfrak{g})$.

Therefore Theorem \ref{thm.butt} above can be seen as a result comparing the translation length functions associated to two cocompact, isometric actions of (the hyperbolic group) $\G = \pi_1(M)$ on the universal covers $(\widetilde{M},\widetilde{\mathfrak{g}})$ and $(\widetilde{M}_\ast,\widetilde{\mathfrak{g}}_\ast)$. 

More generally, given two isometric actions of $\G$ on $(X,d_X)$ and $(X_\ast,d_\ast)$, their \emph{dilation} is given by
\begin{equation*}\label{eq.defdilation}
\Dil(X_\ast,X):=\inf \{\eta >0 \hspace{1mm} \colon \hspace{1mm} \ell_{X_\ast}[g]\leq \eta \cdot \ell_{X}[g] \text{ for all }[g]\in \conj \} \in [0,+\infty].
\end{equation*}
When $\G$ is hyperbolic, $X$ and $X_\ast$ are symmetric and geodesic and the actions of $\G$ on these spaces are proper and cobounded, the dilations $\Dil(X_\ast,X), \Dil(X,X_\ast)$ are also the optimal multiplicative constants for a $\G$-equivariant quasi-isometry between $X$ and $X_\ast$, see \cite[Thm.~1.1]{cantrell.reyes}.

Our next theorem is then a natural generalisation of Theorem \ref{thm.butt}. Recall that the action of a group $\G$ on a possibly asymmetric metric space $(X,d_X)$ is called $D$-\emph{cobounded} ($D\geq 0$) if for any $x,x' \in X$ we can find $g \in \G$ such that $\max\{d_X(gx,x'),d_X(x',gx)\} \le D$.
\begin{theorem}\label{thm.approxLSRvscobounded}
   There exists an absolute constant $K>0$ such that the following holds. Suppose $\G$ is a group acting isometrically on the spaces $X,X_\ast$ such that:
   \begin{itemize}
       \item $X_\ast$ is a $\del$-hyperbolic metric space and the action of $\G$ on this space is non-elementary; and,
       \item $X$ is a (possibly asymmetric) geodesic metric space and the action of $\G$ on $X$ is $D$-cobounded.  
   \end{itemize}
If $L>6D$ then 
\begin{equation*}\label{LSRcoboundedconclusionred}
    \Dil(X_\ast,X)\leq \left(\sup_{0<\ell_X[g]\leq L}\frac{\ell_{X_\ast[g]}}{\ell_X[g]}\right) \cdot \left(\frac{L-2D}{L-6D}\right)+\frac{2K\del}{L-6D}.
   \end{equation*}
\end{theorem}

\begin{remark} In the theorem above, we do not require the action on $X_\ast$ to be proper, cobounded, or even acylindrical. However, $X_\ast$ is an honest metric space (i.e. $d_{X_\ast}$ is symmetric).

\end{remark} 

In the case both $X,X_\ast$ are hyperbolic and geodesic metric spaces and the actions of $\G$ on these spaces are proper and cobounded (so that $\G$ is hyperbolic), Theorem \ref{thm.approxLSRvscobounded} implies the following, which is the coarse analog of Butt's theorem.

\begin{corollary}\label{coro.ineqmetricstructures}
For any $\del,D\geq 0$ there exist positive constants $L_0,C_0$ depending only on $\del$ and $D$ satisfying the following. Let $\G$ a non-elementary hyperbolic group acting properly and $D$-coboundedly on the $\del$-hyperbolic and geodesic metric spaces $X,X_\ast$. If $L>L_0$ and 
\begin{equation*}\label{eq.buttassumptioncoarse}
\al \leq \frac{\ell_{X_\ast}[g]}{\ell_{X}[g]}\leq \beta  \ \ \text{ for all } [g]\in \conj' \text{ with }\ell_{X}[g]\leq L,
\end{equation*}
then 
\begin{equation*}\label{eq.buttconclusioncoarse}
\al \cdot \left(1-\frac{C_0}{L} \right)-\frac{C_0}{L}\leq \frac{\ell_{X_\ast}[g]}{\ell_X[g]}\leq \beta \cdot \left(1+\frac{C_0}{L} \right)+\frac{C_0}{L}
\end{equation*}
for all $[g]\in \conj'$.
\end{corollary}

As an illustrative example, when $\G$ is the fundamental group of a closed hyperbolic surface $S$ and the actions on $X_\ast=X=\Hy^2$ correspond to points in $\calT(S)$, then $\frac{1}{2}\log \Dil(X_\ast,X)$ is precisely the asymmetric Lipschitz distance $d_L(X_\ast,X)$ introduced by Thurston \cite{thurston}. The $D$-cobounded assumption translates to $X$ being in the $\epsilon$-thick part of $\calT(S)$ for some $\epsilon$ depending only on $D$ (cf. Remark \ref{rmk.RGvscobounded}). Theorem \ref{thm.approxLSRvscobounded} then states that $d_L(X_\ast,X)$ can be approximated by looking at quotients $\ell_{X_\ast}[g]/\ell_X[g]$ for $[g]$ representing a closed geodesic with $X$-length at most $L$ and with error term of form $o(L)$.  This is reminiscent of \cite[Thm.~E]{LRT}, in which $d_L(X_\ast,X)$ is approximated by $\ell_{X_\ast}[g]/\ell_X[g]$ for $[g]$ a conjugacy class in a \emph{short marking} of $X$ and with additive error depending only on the topology of $S$ (note that here there are no assumption of $X$ lying on the thick part of $\calT(S)$).

For $S$ and $\G$ as above, a natural extension of Teichm\"uller space is the \emph{quasi-Fuchsian space} $\mathcal{QF}(S)$, which is the space of convex-cocompact representations $\rho:\G \ra \PSL_2(\C)$ up to conjugacy. The convex-cocompact representation $\rho$ induces a natural action of $\G$ on $\Hy^3$ from which we obtain a length function $\ell_\rho: \conj \ra \R$. As an immediate consequence of Theorem \ref{thm.approxLSRvscobounded} we deduce the following for pairs of convex-cocompact representations.

\begin{corollary}\label{coro.qF}
  Let $K$ be the universal constant from Theorem \ref{thm.approxLSRvscobounded}. Let $\G$ be a hyperbolic surface group and $\rho,\rho_\ast:\G \ra \PSL_2(\C)$ convex-cocompact representations inducing length functions. $\ell_\rho,\ell_{\rho_\ast}:\conj \ra \R$. 
If the $\rho$-action of $\G$ on the convex hull of the limit set $\Lam_{\rho(\G)}$ is $D$-cobounded and $L>6D$, then
\begin{equation*}\label{eq.qF}
   \Dil(\rho_\ast,\rho):= \sup_{[g]\in \conj'}{\frac{\ell_{\rho_\ast}[g]}{\ell_\rho[g]}} \leq \frac{L}{L-6D}\cdot \left(\max_{0<\ell_\rho[g]\leq L}\frac{\ell_{\rho_\ast}[g]}{\ell_{\rho}[g]} \right)+\frac{2K\log{4}}{L-6D}. 
\end{equation*}
\end{corollary}

Back to the coarse geometric setting, a very natural and rich source of isometric actions is given by groups acting on their Cayley graphs. When $(X,d_X)=(\Cay(\G,S),d_S)$ is a Cayley graph of $\G$ associated to the (non-necessarily symmetric) generating set $S$, we denote $\ell_{d_S}$ by $\ell_S$ and we prove the following.

\begin{theorem}\label{thm.approxLSRvswordmetric}
   There exists an absolute constant $K>0$ such that the following holds. Let $\G$ be a group with a non-elementary action by isometries on the $\del$-hyperbolic space $X_\ast$. Let $S\subset \G$ be a set that generates $\G$ as a semigroup, so that $\G$ acts isometrically on the Cayley graph $(\Cay(\G,S),d_S)$.
 Then for any integer $L \geq 1$ we have
   \begin{equation*}
    \Dil(X_\ast,\Cay(\G,S))\leq \frac{K\del}{L}+  \sup_{0<\ell_S[g]\leq 2L}{\frac{\ell_{X_\ast}[g]}{\ell_S[g]}}.
   \end{equation*}
\end{theorem}

\begin{remark}
    In the theorem above the set $S$ is not required to be finite or symmetric.
\end{remark}


Many groups admit non-elementary actions on hyperbolic spaces. 
Examples include:
\begin{enumerate}
    \item Non-elementary hyperbolic groups and relatively hyperbolic groups.
    \item Mappings class groups of most oriented finite type surfaces acting on the curve graph \cite{mazur-minsky}.
    \item Outer automorphism groups of free groups acting on the free factor complex \cite{bestvina-feighn.hyp}.
    \item Non-elementary $\CAT(0)$ groups with rank-1 elements \cite{sisto} (see also \cite[Thm.~K]{PSZ}). 
\end{enumerate}

Moreover, many of these groups admit cobounded isometric actions on geodesic metric spaces besides Cayley graphs. For example, mapping class groups act coboundedly on injective metric spaces coming from the hierarchical structure \cite{HHP.coarse,petyt-zalloum}. In addition, if $\G$ is a group acting geometrically on a $\CAT(0)$ cube complex $X$, then $\G$ acts coboundedly on $X$ equipped with any $\ell^p$ metric for $1\leq p\leq \infty$ (see e.g.~\cite{HHP.lp}).


\subsection{Anosov represenations}\label{sec.introanosov}
Our methods are sufficiently flexible that we can apply them to deduce approximate rigidity results for Anosov representations. A representation $\rho: \G \ra \PSL_m(\R)$ is \emph{projective Anosov} (or equivalently, \emph{$1$-dominated} \cite{BPS}) if given a finite generating set $S$ for $\G$ there exist $C,\mu >0$ such that
\begin{equation}\label{eq.defAnosov}
\frac{\sigma_1(\rho(g))}{\sigma_2(\rho(g))} \ge C e^{\mu |g|_S} \ \text{ for all $g\in\G$,}
\end{equation}
where $\sig_i(A)$ represents the $i$th largest singular value of the matrix $A$. See \cite{BPS} for a more detailed introduction to projective Anosov representations.

If $\rho:\G \ra \PSL_m(\R)$ and $\tau: \G \ra \PSL_k(\R)$ are projective Anosov representations, then their dilation is given by 
\begin{equation*}
    \Dil(\rho,\tau):=\sup_{[g]\in \conj'}{\frac{\log \lam_1(\rho(g))}{\log\lam_1(\tau(g))}},
\end{equation*}
where $\lam_1(A)$ denotes the spectral radius of the matrix $A$. We note that the dilation is well-defined. Indeed, we have that $\Dil(\rho,\tau)=\Dil( \G_\rho,\G_\tau)$ where $\G_\rho=(\G,d_\rho)$ and $d_\rho(g,h)=\log \sig_1(g^{-1}h)$, and similarly for $\G_\tau$ (note that $d_\rho$ is possibly asymmetric). For two projective Anosov representations of arbitrary dimension we prove the following, extending Corollary \ref{coro.qF}.

\begin{theorem}\label{thm.approxLSRAnosov}
 For any $m\geq 0$ there exist positive constants $c_m\leq 8\log 2+5 \log m$ and $d_m\leq 2m^3$ with $d_m$ an integer such that the following holds. Suppose $\rho:\G \ra \PSL_m(\R)$ and $\tau: \G \ra \PSL_k(\R)$ are projective Anosov representations (with $m$ not necessarily equal to $k$). Then there exists $\beta>1$ depending only on $\tau$ such that if $L>d_m\beta$ then
\begin{equation*}
    \Dil(\rho,\tau) \leq \frac{c_md_m}{(L-d_m\beta)}+\left(\sup_{0<\log\lam_1(\tau(g))\leq L }{\frac{\log\lam_1(\rho(g))}{\log \lam_1(\tau(g))}} \right) \cdot \left(\frac{L}{L-d_m\beta}\right).
\end{equation*}
\end{theorem}

\subsection{Ideas behind the proofs}
The key ideas behind our proofs are as follows. To obtain bounds on the dilation $\Dil(X_\ast,X)$ as in Theorem \ref{thm.approxLSRvscobounded} we first obtain bounds on dilations of the form $\Dil(X_\ast,\Cay(\G, S))$ (as in Theorem \ref{thm.approxLSRvswordmetric}) where $S \subset \G$ generates $\G$ as a semigroup. To do this we identify $\Dil(X, \Cay(\G,S))$ with the \emph{joint stable length}
\[
\frakD_X(S):=\lim_{n\to \infty}{\frac{\sup_{s_1,\dots,s_n\in S}d_{X_\ast}(x,s_1\cdots s_n x)}{n}}.
\]
The quantity $\frakD_X(S)$ is a geometric version of the well-known joint spectral radius associated to products of matrices. This identification is useful as it allows us to apply an inequality of Breuillard and Fujiwara (see Theorem \ref{thm.breufuji}) that gives us uniform control over $\frakD_{X_\ast}(S)$ (and hence $\Dil(X_\ast,\Cay(\G, S))$) in terms of $\sup_{s \in S^2}\ell_{X_\ast}[s]$. We then use the fact that the displacement function $d_{X}$ associated to the action of $\G$ on $X$ can be approximated by word metrics (see Lemma \ref{lem.HDLFapproxword}). This allows us to obtain bounds on $\Dil(X_\ast,X)$ using the inequality
\[
\Dil(X_\ast,X) \le \Dil(X_\ast, \Cay(\G,S)) \, \Dil( \Cay(\G,S),X)
\]
along with an appropriate choice of $S$ and the arguments mentioned above. We follow this line of reasoning to prove both Theorem \ref{thm.approxLSRvscobounded} and \ref{thm.approxLSRvswordmetric}. To deduce Theorem \ref{thm.approxLSRAnosov} we follow a similar argument but replace the use of the Breuillard-Fujiwara inequality with an inequality due to Bochi (Theorem \ref{thm.bochi}) for matrix products. In the last section we deduce Theorem \ref{thm.butt} from Theorem \ref{thm.approxLSRvscobounded}. This amounts to realising the coarse geometric constants that appear in Theorem \ref{thm.approxLSRvscobounded} in terms of the curvature and injectivity radius constants comings from the Riemannian manifold.

\subsection*{Organisation}
The article is organised as follows. In Section \ref{sec.prelim} we introduce preliminary material concerning metric spaces and groups acting on them. We then in Section \ref{sec.qr} prove the theorems presented in sections \ref{sec.introhyperbolic} and \ref{sec.introanosov}. In Section \ref{sec.buttthm} we deduce Theorem \ref{thm.butt}.

\subsection*{Acknowledgements}
We are grateful to Karen Butt, Ralf Spatzier for helpful comments and suggestions, and we are indebted to Abdul Zalloum for careful reading and feedback on an earlier draft.


\section{Preliminaries}\label{sec.prelim}

\subsection{Pseudo metric spaces} \label{sec.hyp}
Recall that a \emph{metric} $d$ on the space $X$ is a function $d:X \times X \ra \R$ satisfying: 
   \begin{enumerate}
        \item point separation: $d(x,y)=0$ implies $x=y$;
        \item positivity: $d(x,y)\geq 0$ for all $x,y\in X$, and $d(x,x)=0$ for all $x\in X$;
        \item the triangle inequality: $d(x,z)\leq d(x,y)+d(y,z)$ for all $x,y,z\in X$; and,
        \item symmetry: $d(x,y)=d(y,x)$ for all $x,y\in X$.
    \end{enumerate}
If $d$ only satisfies (2),(3), and (4) above then it is called a \emph{pseudo metric}, and $d$ is a \emph{possibly non-symmetric} metric (resp. pseudo metric) if it satisfies (1), (2) and (3) (resp. (2) and (3)), but not necessarily (4). 

We say that two functions $\phi, \psi : X \times \G \to \R$ are \emph{roughly similar} if there exist constants $\tau, C>0$ such that
\[
|\phi(x,y) - \tau \psi(x,y)| \le C \ \text{ for all $x,y \in \G$}.
\]
If $\psi, \phi$ are roughly similar with $\tau =1$ we say that they are \emph{roughly isometric}. In a similar way, we say that $\phi$ and $\psi$ are \emph{quasi-isometric} if there exists $\lam>1$ such that 
\[\lam^{-1}\psi(x,y) -\lam \leq \phi(x,y)\leq \lam \psi(x,y)+\lam \text{ for all} x,y\in X.\]
Note that roughly similar functions are necessarily quasi-isometric.

The \emph{Gromov product} of a function $\psi: X\times X \ra \R$ is the assignment 
\[(x|y)^\psi_z:=\frac{1}{2}(\psi(x,z) + \psi(z,y) - \psi(x,y)) \text{ for $x,y,z \in X$}. \]
A possibly asymmetric pseudo metric $d$ on $X$ is called $\del$-\emph{hyperbolic} ($\del\geq 0$) if
\[
(x|y)^d_z \ge \min\{(x|w)^d_z, (y|w)^d_z \} - \delta \text{ for all} x,y,z,w \in X,\]
and $d$ is \emph{hyperbolic} if it is $\del$-hyperbolic for some $\del$. 

\subsection{Group actions and hyperbolicity}

If $X=(X,d_X)$ is a possibly asymmetric pseudo metric space and $\G$ is a group with identity element $o$, an action of $\G$ on $X$ is \emph{isometric} if $d_X(gx,gy)=d_X(x,y)$ for all $x,y\in X$ and $g\in \G$. In this case, the \emph{stable translation length} function of this action is given by
\begin{equation*}
    \ell_X[g]=\ell_{d_X}[g]:=\lim_{n\to \infty}{\frac{1}{k}d_X(x,g^n x)} \hspace{2mm} \text{for } [g]\in \conj,
\end{equation*}
which is well-defined by subadditivity and independent of the base point $x\in X$. 

As in the introduction, the \emph{dilation} of two isometric actions of $\G$ on the possibly asymmetric pseudo metric spaces $X, X_\ast$ is then the infimum
\[
\Dil(X_\ast,X):=\inf \{\eta >0 \hspace{1mm} \colon \hspace{1mm} \ell_{X_\ast}[g]\leq \eta \cdot \ell_{X}[g] \text{ for all }[g]\in \conj \} \in [0,+\infty]
\]
and declare $\Dil(X_\ast,X)$ to be $+\infty$ if no such $\eta$ exists. When $X=(X,d), X_\ast=(X,d_\ast)$ and $X$ is understood, we often write $\Dil(d_\ast,d)=\Dil(X_\ast,X)$.

Note that $\Dil(d_\ast,d)$ is always finite when $d$ and $d_\ast$ are quasi-isometric, possibly asymmetric pseudo metrics on a given space $X$. In addition, we have the submultiplicative property 
\[
\Dil(X_{\ast \ast},X)\leq \Dil(X_{\ast\ast},X_\ast)\Dil(X_{\ast},X)
\]
valid for all isometric actions of $\G$ on the possibly asymmetric pseudo metric spaces $X,X_\ast,X_{\ast\ast}$.

Natural pseudo metrics on groups are word metrics. If $S$ is a (non-necessarily symmetric) set of the group $\G$ that generates it as a semigroup, then the right word metric
$d_S(g,h):=|g^{-1}h|_S$ is a possibly asymmetric metric on $\G$. We abbreviate the stable translation length function $\ell_{d_S}$ by $\ell_S$ and if $\G$ also acts isometrically on $X$, we use the notation $\Dil(X,S)=\Dil(X,d_S)$ and $\Dil(S,X)=\Dil(d_S,X)$. A similar notation is used for the dilation comparing $d_S$ and left-invariant, possibly asymmetric pseudo distances on $\G$.

A group $\G$ is \emph{non-elementary hyperbolic} if it is non-virtually cyclic, finitely generated, and any word metric $d_S$ on $\G$ associated to a finite, symmetric generating set $S$ is hyperbolic. Under this assumption we write $\Dc_\G$ for the collection of pseudo metrics on $\G$ that are hyperbolic, quasi-isometric to a word metric (for a finite generating set), and $\G$-invariant. The \emph{space of metric structures} $\scrD_\G$ is the quotient of $\calD_\G$ under rough similarity. There is a natural metric $\Del$ on $\scrD_\G$ given by 
\begin{equation}\label{eq.defDel}
\Del([d],[d_*]):=\log  \left(\Dil(d,d_*) \Dil(d_*,d) \right)
\end{equation}
where $[d]$ is the rough similarity equivalence class of $d\in \calD_\G$, see \cite{reyes}.

\begin{definition}
    If $\G$ is non-elementary hyperbolic, a \emph{hyperbolic distance-like function} on $\G$ is a possibly asymmetric pseudo metric  $\psi:\G \times \G \ra \R$ that is $\G$-invariant and satisfies the following. For any $d\in \calD_\G$ and $C\geq 0$ there exists $D\geq 0$ such that $(x|y)_{w}^{d}\leq C$ implies $(x|y)_w^\psi \leq D$ for all $x,y,w\in \G$.
\end{definition}
A related notion to hyperbolic-like functions is the family of $\G$-invariant \emph{tempered potentials} that was considered by Cantrell-Tanaka \cite[Def.~2.5]{cantrell.tanaka.1}. These are $\G$-invariant functions satisfying property (4) above and such that their Gromov products are `stable at infinity'. 

It is immediate that pseudo metrics belonging to $\calD_\G$ are hyperbolic distance-like functions. Also, if $S$ is a non-necessarily symmetric finite generating set of $\G$, then the word metric $d_S$ is a hyperbolic distance-like function.
Similarly, if $\rho:\G \ra \PSL_{m}(\R)$ is a projective Anosov representation (recall \eqref{eq.defAnosov}), then
\[
\psi^\rho(g,h):=\log \sig_1(g^{-1}h) \ \text{ for $g,h \in \G$}
\]
is a hyperbolic distance-like function \cite[Lem.~3.10]{cantrell.reyes}. These examples are all quasi-isometric to pseudo metrics in $\Dc_\G$. However, hyperbolic distance-like functions need not be proper. 

\begin{definition}
    For $\G$ non-elementary hyperbolic, let $\ov\Dc_\G$ be the set of all left-invariant pseudo metrics on $\G$ that have non-constant translation length function and such that there exist $\lambda >0$ and $ d_\ast\in \Dc_\G$ satisfying $(g|h)_o^d \le \lambda(g|h)_o^{d_\ast} + \lambda$ for all $g,h \in \G$. We write $\ov\scrD_\G$ for the quotient of $\ov\Dc_\G$ under the equivalence of rough similarity. Points in $\ov\scrD_\G$ are called \emph{metric structures}.
\end{definition}
 It follows immediately from the definition that elements in $\ov\Dc_\G$ are hyperbolic distance-like functions and that there is a natural inclusion $\scrD_\G \subset \ov \scrD_\G$.

 If $\G$ acts isometrically on the metric space $X$ and $x\in X$ then $d_X^x(g,h):=d_X(gx,hx)$ defines a left-invariant pseudo metric on $\G$. If this pseudo metric belongs to $\ov\calD_\G$ for some (equivalently, for any) $x\in X$, we abuse notation and let $[X]$ denote the metric structure in $\ov\scrD_\G$ induced by $d_X^x$. This metric structure is independent of the point $x$.
 
 It was shown in \cite[Thm.~1.6]{cantrell.reyes} that the following actions induce points in $\ov\scrD_\G$. 
\begin{enumerate}
    \item Actions on coned-off Cayley graphs for finite, symmetric generating sets, where we cone-off a finite number of quasi-convex subgroups of infinite index.
    \item Non-trivial Bass-Serre tree actions with quasi-convex vertex stabilizers of infinite index.
    \item Non-trivial small actions on $\R$-trees, when $\G$ is a surface group or a free group.
\end{enumerate}
Pseudo metrics corresponding to these actions need not be proper, i.e. given $d \in \ov\Dc_\G$ is possible that there is $T >0$ such that the set $\{g \in \G : d(o,g)  \leq T\}$ is infinite. 

We will need two properties about hyperbolic distance-like functions. The first one is that these functions are always roughly geodesic.
\begin{definition}
If $X$ is a non-empty set, a non-negative function $\psi:X \times X \ra \R$ is \emph{$\al$-roughly geodesic} ($\al\geq 0$) if for all $x,y\in X$ there exists a sequence $x=x_0,\dots,x_n=y$ in $X$ such that
$$j-i-\al \leq \psi(x_i,x_j)\leq j-i+\al \text{ for all }0\leq i\leq j \leq n.$$
\end{definition}

\begin{proposition}[{\cite[Lem.~3.2]{cantrell.reyes}}]\label{prop:RGforHDLF} 
    If $\G$ is non-elementary hyperbolic and $\psi:\G \times \G \ra \R$ is a hyperbolic distance-like function, then there exists some $\al\geq 0$ such that $\psi$ is $\al$-roughly geodesic.
\end{proposition}

The second property is that hyperbolic distance-like functions can be approximated by word metrics. This assertion also holds for arbitrary cobounded isometric actions on roughly geodesic spaces. Part (1) of the next result follows exactly as in the proof of \cite[Lem.~5.1]{reyes} which shows that pseudo metrics in $\Dc_\G$ can be approximated by word metrics (the assumption of hyperbolicity is not needed), whereas part (2) is exactly \cite[Lem.~4.6]{BCGS} after noting that the assumptions of properness and symmetry are unnecessary. We leave the details to the reader. 
\begin{lemma}\label{lem.HDLFapproxword} 
Let $\G$ act isometrically on the possibly asymmetric metric space $X$ and fix $x\in X$. 
\begin{enumerate}
    \item Suppose $X$ is geodesic and the action of $\G$ is $D$-cobounded. For an integer $n\geq 1$, let $S_n=\{g \in \G \mid d_X(x, gx) \leq (n+2)D\}$. Then $S_n$ generates $\G$ as a semigroup and for all $g\in \G$ we have 
    $$ nD|g|_{S_n}-nD \leq d_X(x,gx) \leq (n+2)D|g|_{S_n}. $$
    \item Suppose $X=(\G,\psi)$ where $\G$ acts by the left on $X$ and $\psi$ is $\al$-roughly geodesic. For $n>\alpha+1$, let $S_n:=\{g \in \G \mid \psi(x, gx) \leq n\}$. Then $S_n$ generates $\G$ as a semigroup and for all $g \in \G$ we have
    $$ (n-\alpha-1)|g|_{S_n}-(n-1) \leq \psi(x,gx) \leq n|g|_{S_n}. $$
\end{enumerate}
\end{lemma}


\begin{remark}
    (1) In both cases above, we do not require the action to induce a quasi-isometric embedding of $\G$ into $X$, so the sets $S_L$ can be infinite. \\
    
    \noindent (2) From the second part of the lemma above we deduce that $\psi$ is quasi-isometric to $d_{S_n}$ for all $n$ large enough. When $\G$ is hyperbolic and $\psi$ is a hyperbolic distance-like function, this implies that if $\psi$ is not quasi-isometric to pseudo metrics in $\calD_\G$, then the sets $S_n$ must be infinite. Therefore, a pseudo metric in $\ov\calD_\G$ belongs to $\calD_\G$ if and only if it is proper.
\end{remark}

\section{Approximate quantitative rigidity}\label{sec.qr}

In this section we prove our approximate length spectrum rigidity theorems in Sections \ref{sec.introhyperbolic} and \ref{sec.introanosov}.

Let $\G$ be a (non-necessarily hyperbolic) group acting isometrically on the possibly asymmetric metric space $X=(X,d_X)$. For a set $S\subset \G$ we let $S^n$ be the set of all products of $n$ elements in $S$. If  $x\in X$ we set $d_X(x,S)=\sup_{s\in S}{d_X(x,sx)}$. We say that $S$ is \emph{bounded} for the action of $\G$ on $X$ if $d_X(x,S)$ is finite for some (and hence any) $x\in X$.

If $S\subset \G$ is bounded then its \emph{joint stable length} is given by the limit
\[
\frakD_X(S):=\lim_{n\to \infty}{\frac{d_X(x,S^n)}{n}},
\]
which is finite for every finite set $S\subset \G$ and independent of $x\in X$. The joint stable length is an analog of the joint spectral radius for sets of matrices and first appeared in \cite{or} (see also \cite{b.f}). 

When $X$ is a hyperbolic metric space and the action of $\G$ is non-elementary, the relation between the joint stable length and the dilation is given by the following result, for which we recall the notation $\Dil(X_\ast,S)=\Dil(X_\ast,\Cay(\G, S))=\Dil(X,d_S)$ when $S$ generates $\G$. 

\begin{proposition}\label{prop.dil=jsr}
Let $\G$ act isometrically on the possibly asymmetric metric space $X$ and let $S\subset \G$ be a subset that generates $\G$ as a semigroup. Then $\Dil(X,S)\leq \frakD_X(S)$, so that $S$ being bounded implies that $\Dil(X,S)$ is finite.

Conversely suppose that either:
\begin{enumerate}
    \item $X$ is a hyperbolic metric space and the action of $\G$ is non-elementary; or, 
    \item $X=(\G,\psi)$ where $\G$ is hyperbolic and acts by the left on $X$ and $\psi$ is a hyperbolic distance-like function. 
\end{enumerate}
Then $\Dil(X,S)$ being finite implies that $S$ is bounded and moreover we have $\Dil(X,S)=\frakD_X(S)$.
\end{proposition}

We will need the following lemma, which follows from \cite[Cor.~3.12]{cantrell.reyes} when $X=(\G,\psi)$ and $\psi$ is a hyperbolic distance-like function on the hyperbolic group $\G$. When $X$ is hyperbolic and the action is non-elementary, it follows, for instance, from \cite[Lem.~4.7]{WXY}.

\begin{lemma}\label{lem.gam_x}
    Let $\G$ and $X$ be as in the assumptions (1) or (2) of Proposition \ref{prop.dil=jsr}. Then for any $x\in X$ there exist $C>0$ depending on $x$ and a finite set $F \subset \G$ such that for any $g\in \G$ we have
\begin{equation*}\label{eq:AMSproperty}
    d_X(x,gx)\leq \max_{f\in F}{\ell_X[gf]}+C.
\end{equation*}    
\end{lemma}

\begin{proof}[Proof of Proposition \ref{prop.dil=jsr}] We fix $x\in X$ and let $|\cdot|_S$ denote the word length on $\G$ with respect to $S$. Let $g\in \G$ and for $k>0$ we set $n(k)=|g^k|_S$, so that $g\in S^{n(k)}$. If $\ell_X[g]>0$ then $n(k)$ tends to infinity as $k\to \infty$ and we have
\begin{align*}
    \ell_X[g]=\lim_{k\to \infty} {\frac{d_X(x,g^kx)}{k}} & \leq \limsup_{k \to\infty}{\frac{d_X(x,S^{n(k)})}{k}} \\ & \leq \limsup_{k \to\infty}{\left(\frac{d_X(x,S^{n(k)})}{n(k)} \cdot \frac{|g^k|_S}{k}\right)}=\frakD_X(S) \, \ell_S[g].
\end{align*}

Since this holds for all $g \in \G$ with $\ell_S[g]>0$ we obtain $\Dil(X, S) \leq \frakD_X(S)$.

For the reverse inequality, assume that $\G$ and $X$ satisfy either (1) or (2) and let $C>0$ and $F\subset \G$ be given by Lemma \ref{lem.gam_x}.  Then for $n>0$ and $g\in S^n$ we have
\begin{align*}
    d_X(x,gx) &\leq \max_{f\in F}{\ell_X[gf]}+C \\ & \leq  \Dil(X,S)\max_{f\in F}{\ell_S[gf]}+C \\
    & \leq \Dil(X,S)(|g|_S+\max_{f\in F}{|f|_S})+C \\ & \leq n\Dil(X,S)+(\Dil(X,S)\max_{f\in F}{|f|_S}+C).
\end{align*}

This last inequality holds for any $g\in S^n$, and since $F$ is finite we obtain that $S$ is bounded if $\Dil(X,S)$ is finite. Also, after taking supremum over $g\in S^n$, dividing by $n$ and letting $n$ tend to infinity we get $\frakD_X(S)\leq \Dil(X,S)$, as desired.
\end{proof}

Now we move to the proofs of Theorems \ref{thm.approxLSRvscobounded} and \ref{thm.approxLSRvswordmetric}. For this we need the following bound for the joint stable length of isometric actions on Gromov hyperbolic spaces due to Breuillard and Fujiwara. 

\begin{theorem}[Breuillard-Fujiwara {\cite[Thm.~1.4]{b.f}}] \label{thm.breufuji}
There exists a universal constant $K>0$ such that if $X$ is a $\del$-hyperbolic metric space and $S$ is a bounded set of isometries of $X$, then
\begin{equation}\label{eq.breufuji}
\frakD_X(S) \leq K \delta+\frac{1}{2} \sup_{s\in S^2 }\ell_X\left[s\right] \leq K \delta+\frakD_X(S).
\end{equation}    
\end{theorem}

\begin{remark}
(1) The statement in \cite[Thm.~1.4]{b.f} assumes $S$ to be finite. However, it suffices for $S$ to be bounded as we can easily verify $\frakD_X(S)=\sup\{\frakD_X(F) |F\subset S \text{ finite}\}$.\\

\noindent (2) In \cite{b.f}, the inequality \eqref{eq.breufuji} is proven in the case that $X$ is geodesic. However, this assumption can be overcome by considering the isometric embedding of $X$ into its \emph{injective hull} $E(X)$ \cite{lang}. The isometric action of $\G$ on $X$ extends to an isometric action on $E(X)$ under this injection, and if $X$ is $\del$-hyperbolic, then $E(X)$ is also $\del$-hyperbolic \cite[Prop.~1.3]{lang}. The inequality \eqref{eq.breufuji} for $E(X)$ then implies the inequality for $X$, since the quantities $\frakD_\bullet(S)$ and $\max_{s\in S^2 }\ell_\bullet\left[s\right]$ are preserved under $\G$-equivariant isometric embeddings. However, the assumption that  $X$ is geodesic is crucial to obtain uniform comparisons between $\frakD_X(S)$ and the \emph{joint displacement} $L_{X}(S)=\inf_{x\in X}{d_X(x,S)}$, which is not required in this paper.
\end{remark}

\begin{proof}[Proof of Theorem \ref{thm.approxLSRvswordmetric}]
        Let $\G$ be a group acting non-elementarily on the $\del$-hyperbolic metric space $X_\ast$ and let $S\subset \G$ generate $\G$ as a semigroup so that it is bounded with respect to this action. Let $K$ be the universal constant given by Theorem \ref{thm.breufuji}, and for $L \geq 1$ set 
    $A_L:=\sup_{0<\ell_S[g]\leq L}\frac{\ell_{X_\ast[g]}}{\ell_S[g]}$. 

    If $L$ is an integer, by Proposition \ref{prop.dil=jsr} and Theorem \ref{thm.breufuji} we have 
\begin{align*}
\Dil(X_\ast,S)=\frakD_{X_\ast}(S)  =\frac{\frakD_{X_\ast}(S^L)}{L}  \leq \frac{K\del}{L}+\frac{1}{2L}\sup_{s\in S^{2L}}{\ell_{X_\ast}[s]}.
\end{align*}   

Now, if $s\in S^{2L}$ and $\ell_S[s]=0$, then $\ell_{X_\ast}[s]=0$. Otherwise we have $\ell_{X_\ast}[s]\leq A_{2L}\ell_{S}[s]\leq 2L A_{2L}$, and hence 
\begin{equation*}
 \Dil(X_\ast,S)\leq \frac{K \del }{L}+A_{2L},
\end{equation*}
as desired.
\end{proof}

\begin{proof}[Proof of Theorem \ref{thm.approxLSRvscobounded}]
Let $\G$ act isometrically on $X_\ast$ and $X$ as in the assumptions of the theorem and let $K$ be the constant from Theorem \ref{thm.breufuji}. We can further assume that $\Dil(X_\ast,X)$ is finite.


Let $L>6D$ be arbitrary and consider $n:=\lfloor L/2D\rfloor -2$, which is a positive integer satisfying $2(n+2)D \leq L$. For a fixed point $x\in X$ consider the set $S=S_{n}:=\left\{g \in \Gamma \hspace{1mm}\colon \hspace{1mm} d_X(x, gx) \leq (n+2)D\right\}$. Then $S$ is bounded for the action on $X$ by Lemma \ref{lem.HDLFapproxword} (1) and moreover we have
\begin{equation}\label{eq.ineqdilations}
    \Dil(X,S)\leq \frakD_X(S) \leq  (n+2)D \ \ \text{ and } \ \ \Dil(S,X)\leq (nD)^{-1}.
\end{equation}
This implies that $\Dil(X_\ast,S)\leq \Dil(X_\ast,X)\Dil(X,S)$ is finite and hence $S$ is bounded for the action on $X_\ast$ by Proposition \ref{prop.dil=jsr}.

Therefore, by Theorem \ref{thm.breufuji} and Proposition \ref{prop.dil=jsr} we obtain
\begin{align*}
\Dil(X_\ast, X) & \leq \Dil(X_\ast, S) \, \Dil(S, X) \\
&=\frakD_{X_\ast}(S) \Dil(S, X) \\
&\leq\left(K\del+\frac{1}{2}\sup_{s\in S^2}\ell_{X_\ast}[s]\right) \, \Dil(S, X).
\end{align*}


Now note that if $s\in S^2$ then $\ell_{X}[s]\leq \Dil(X,S)\ell_S[x]\leq 2(n+2)D\leq L$, and hence Proposition \ref{prop.dil=jsr} gives us
$$ \frac{1}{2}\sup_{s\in S^2}\ell_{X_\ast}[s]\leq \left(\sup_{0<\ell_X[g]\leq L}\frac{\ell_{X_\ast[g]}}{\ell_X[g]}\right)  \cdot \frac{1}{2}\sup_{s\in S^2}\ell_X[s]\leq \left(\sup_{0<\ell_X[g]\leq L}\frac{\ell_{X_\ast[g]}}{\ell_X[g]}\right)  \cdot \frakD_X(S).$$
From this and \eqref{eq.ineqdilations} we deduce
\begin{align*}
\Dil(X_\ast, X) & \leq \left(K \del+\left(\sup_{0<\ell_X[g]\leq L}\frac{\ell_{X_\ast[g]}}{\ell_X[g]}\right)  \cdot \frakD_X(S)\right)\Dil(S, X)\\ 
& \leq \left(\sup_{0<\ell_X[g]\leq L}\frac{\ell_{X_\ast[g]}}{\ell_X[g]}\right)  \cdot \left(\frac{(n+2)D}{nD}\right)+\frac{K\del}{nD} \\
& \leq \left(\sup_{0<\ell_X[g]\leq L}\frac{\ell_{X_\ast[g]}}{\ell_X[g]}\right)  \cdot \left( \frac{L-2D}{L-6D} \right)+\frac{2K\del}{L-6D},
\end{align*}
where in the last inequality we used that $n\geq L/2D-3>0$. This concludes the proof of the theorem.
\end{proof}

We end this section with the proof of Theorem \ref{thm.approxLSRAnosov}, which is a straightforward consequence of Proposition \ref{prop:RGforHDLF} and the following theorem applied to $\psi=\psi^\tau$ the hyperbolic distance-like function induced by $\tau$.

\begin{theorem}\label{thm.approxmatrices}
  For any $m\geq 1$ there exist positive constants $c_m\leq 8\log 2+5 \log m$ and $d_m\leq 2m^3$ with $d_m$ an integer such that the following holds. Suppose $\rho:\G \ra \PSL_m(\R)$ is a projective Anosov representation (so that $\G$ is hyperbolic), and let $\psi: \G \times \G \ra \R$ be a hyperbolic distance-like function that is $\al$-roughly geodesic and quasi-isometric to pseudo metrics in $\calD_\G$. Assume that there $L>d_m(\al+1)$ and $\eta>0$ satisfy
\begin{equation*}\label{eq.assumpAnosov}
       \log{\lam_1(\rho(g))}\leq \eta \cdot \ell_{\psi}[g] \text{ for all }[g]\in \conj \text{ with }\ell_{\psi}[g]\leq L.
   \end{equation*}
   Then
\begin{equation*}
\Dil(\rho,\psi):=\Dil(\psi^\rho,\psi)\leq \frac{c_md_m}{(L-d_m(\al+1))}+\eta \cdot \left(\frac{L}{L-d_m(\al+1)}\right).
\end{equation*}
\end{theorem}

To prove this theorem we cannot directly apply Theorem \ref{thm.approxLSRvscobounded},
since for an Anosov representation $\rho:\G \ra \PSL_m(\R)$ the induced hyperbolic distance-like function $\psi^\rho$ is in general not  symmetric. Instead, we replace Theorem \ref{thm.breufuji} by
estimates due to Bochi \cite{bochi} on the \emph{joint spectral radius} $\frakR(S)$ of a bounded set $S\subset M_m(\R)$ of real $m\times m$ matrices. Recall that
\[
\frakR(S):=\lim_{n\to \infty}{\sup_{A\in S^n} \sig_1(A)}^{1/n},
\]
which is well-defined by submultiplicativity. Here $S^n\subset M_m(\R)$ is the set of $n$-fold products of matrices in $S$. By Proposition \ref{prop.dil=jsr} we have that if $\rho:\G \ra \PSL_m(\R)$ is a projective Anosov representation with hyperbolic distance-like function $\psi^\rho$ and $S\subset \G$ is a finite generating set, then $\Dil(\rho,S)=\Dil(\psi^\rho , d_S)=\log\frakR(\rho(S)).$ 

The main estimate we require is the following, where the upper bounds for $c_m$ and $d_m$ follow from \cite[Thm.~5]{breuillard}.
\begin{theorem}[Bochi {\cite[Thm.~B]{bochi}}]\label{thm.bochi}
   For all $m\geq 1$ there exist $0\leq c_m \leq 8\log 2+5 \log m$ and a positive integer $d_m\leq 2m^3$ such that for any bounded set $S\subset M_{m}(\R)$ we have
    \begin{equation*}
        \frakR(S)\leq e^{c_m} \max_{1\leq j \leq d_m}{\sup_{A\in S^j}\lam_1(A)^{1/j}}.
    \end{equation*}
\end{theorem}


With Theorem \ref{thm.bochi} at our disposal, the proof of Theorem \ref{thm.approxmatrices} goes in the same way as the proof of Theorem \ref{thm.approxLSRvscobounded}, now using Lemma \ref{lem.HDLFapproxword} (2). Details are left to the reader. \qedhere

\section{Recovering Butt's theorem}\label{sec.buttthm}

In this section we prove Theorem \ref{thm.butt},
for which we introduce some terminology. If $M$ is a closed Riemannian manifold, then its \emph{simplicial volume} $\|M\|$ is a real number that describes the simplicial complexity of $M$ \cite[Sec.~0.2]{gromov.volume}. This number is a homotopy invariant, so if $M$ and $M_\ast$ are homotopy equivalent then  $\|M\|=\|M_\ast\|$. When $M, M_\ast$ are negatively curved, this is equivalent to $\pi_1(M)$ being isomorphic to $\pi_1(M_\ast)$, and in this case we write $\|\G\|=\|M\|$ for $\G=\pi_1(M)$. 

If $\G$ is a non-elementary hyperbolic group, recall that the space $\scrD_\G$ of metric structures is the quotient of $\calD_\G$ under rough similarity and that $\Del$ is the metric on $\scrD_\G$ given by \eqref{eq.defDel}. If $\del,\al\geq 0$ then $\scrD_\G^{\del,\al}\subset \scrD_\G$ denotes the subset of all the metric structures represented by a $\del$-hyperbolic and $\al$-roughly geodesic pseudo metric in $\calD_\G$ with \emph{exponential growth rate} 1 (see \cite[Sec.~2]{reyes} for details). Similarly, for $D\geq 0$ we let $\scrD_\G^{\del,(D)}$ denote the space of all metric structures induced by a proper and $D$-cobounded action of $\G$ on a geodesic and $\del$-hyperbolic metric space with exponential growth rate 1. 

\begin{remark}\label{rmk.RGvscobounded}
We can verify that  $\scrD_\G^{\del,(D)}\subset \scrD_\G^{\del,2D+1}$ for all $\del,\al\geq 0$ and $\G$.  The argument for this is as in the proof of \cite[Lem.~4.6]{BCGS}. We leave the details to the reader.
\end{remark}

If $(M,\gm)$ is a closed Riemannian manifold and $\G=\pi_1(M)$ we let $\widetilde{M}_{\gm}=(\widetilde{M},d_{\tilde{\gm}})$ be the universal cover of $M$ equipped with the ($\G$-invariant) Riemannian distance induced by the lift $\tilde{\gm}$ of $\gm$. The Deck action of $\G$ on $\widetilde{M}_{\gm}$ induces the metric structure $[\widetilde{M}_\gm]\in \scrD_\G$.

\begin{lemma}\label{lem.bounddifgeometry}
For any $n\geq 2$ there exists a constant $C_n>0$ such that the following holds. Let $(M,\gm)$ be a closed Riemannian manifold with sectional curvatures contained in the interval $[-\Lam^2,-\lam^2]\subset \R_{<0}$,  injectivity radius $i_\gm$ and fundamental group $\G=\pi_1(M)$. Then
\begin{enumerate}
    \item $\textnormal{Vol}(M,\gm)\leq C_n\|\G\|\lam^{-n}$;
    \item $\textnormal{Diam}(M,\gm)\leq C_n\|\G\|\lam^{-n}i_\gm^{1-n}$;
    \item $\widetilde{M}_{\gm}$ is $\lam^{-1}\log{2}$-hyperbolic and the Deck action of $\G$ on $\widetilde{M}_{\gm}$ is $C_n\|\G\|\lam^{-n}i_\gm^{1-n}$-cobounded; and,
    \item the exponential growth rate $v_{\gm}$ of $\widetilde{M}_{\gm}$ satisfies $\frac{\lam}{C_n}\leq v_{\gm}\leq \Lam(n-1).$
\end{enumerate}
Therefore, we have $[\widetilde{M}_{\gm}]\in \scrD_\G^{\del,(D)}$ for
$$\del=\Lam \lam^{-1}(n-1)\log{2} \hspace{2mm}\text{ and }\hspace{2mm} D=\Lam(n-1)C_n\|\G\|\lam^{-n}i_\gm^{1-n}.$$
\end{lemma}

\begin{proof}
   Item (1) is exactly Gromov-Thurston's theorem (see e.g.~\cite[Sec.~0.3]{gromov.volume}), and Item (2) follows by combining Item (1) and \cite[Cor.~15]{crocke.isopineq}. For Item (3), note first that $\lam d_{\gm}=d_{\lam^2 \gm}$, and that $(M,\lam^2\gm)$ has sectional curvatures bounded above by -1. Then $(\widetilde{M},d_{\lam^2\tilde{\gm}})$ is $\CAT(-1)$, and hence $\widetilde{M}_{\gm}$ is $\log{2}$-hyperbolic \cite[Thm.~4.2 \& Thm.~5.1]{nica-spakula}. The bound for the codiameter follows directly from Item (2).

   The lower bound in Item (4) follows from Item (1) and the inequality $$\|\G\|\leq C_n'\textnormal{Vol}(M,\gm){v_\gm}^n,$$ where $C_n'>0$ depends only on $n$ \cite[Sec.~2.5]{gromov.volume}. For the upper bound, note that sectional curvatures bounded below by $-\Lam^2$ imply that the Ricci curvature is bounded below by $-(n-1)\Lam^2$. Then by Bishop-Gromov's inequality \cite[Lem.~7.1.3]{petersen}, $v_\gm$ is bounded above by $\Lam(n-1)$, which equals the exponential growth rate of the $n$-dimensional Riemannian manifold of constant sectional curvature equal to $-\Lam^2$. 

   The last assertion of the lemma follows from Items (3) and (4).
\end{proof}

\begin{corollary}\label{coro.firstineq}
    Let $(M,\gm),(M_\ast,\gm_\ast)$ be closed Riemannian manifolds of dimension $n$, let $\G$ be the fundamental group of $M$, and let $f_\ast:\G \ra \pi_1(M_\ast)$ be an isomorphism. Assume that their sectional curvatures are bounded above by  $-\lam^2<0$ and let $i_\gm,i_{\gm_\ast}$ denote the injectivity radii of $(M,\gm)$ and $(M_\ast,\gm_\ast)$ respectively. Fix $\eta_0>0$, and let $0< \eta\leq \eta_0$ and $L>0$ be such that 
    \begin{equation}\label{eq.ellRiemannian}
\ell_{\gm}[g]\leq \eta \ell_{\gm_\ast}[f_\ast g] \text{ for all } [g]\in \conj \text{ with }\ell_{\gm_\ast}[f_\ast g]\leq L.
   \end{equation}
    \begin{enumerate}
        \item If $L\geq 2\eta^{-1}_{0}i_\gm$, then $i_{\gm_\ast}\geq\eta_0^{-1}i_\gm$.
        \item If $L>\hat{L}_\ast:=6C_n \|\G\|\lam^{-n}{i_{\gm_\ast}}^{1-n}$, then $$\frac{\ell_\gm[g]}{\ell_{\gm_\ast}[f_\ast g]}\leq \eta \left(\frac{L}{L-\hat{L}_\ast}\right)+ \frac{2K\lam^{-1}\log{2}}{L-\hat{L}_\ast}$$
        for all $[g]\in \conj'$, where $K$ and $C_n$ are the constants from Theorem \ref{thm.approxLSRvscobounded} and Lemma \ref{lem.bounddifgeometry} respectively. 
    \end{enumerate}
\end{corollary}
\begin{remark*}
    Note that the corollary does not require a lower bound for the sectional curvatures.
\end{remark*}
\begin{proof}
To prove (1), note that $i_\gm=\frac{1}{2}\inf_{[g]\in \conj'}{\ell_{\gm}[g]}$, with a similar formula valid for $i_{\gm_\ast}$ \cite[p.~258]{petersen}. Suppose for the sake of contradiction that $i_{\gm_\ast}<\eta_0^{-1}i_\gm$. Then there exists $[g]\in \conj'$ such that $i_{\gm_\ast}=\ell_{\gm_\ast}[f_\ast g]/2$, and since $L\geq 2\eta_{0}^{-1}i_\gm$ we have $$i_{\gm}\leq \ell_{\gm}[g]/2\leq \eta_0\ell_{\gm_\ast}[f_\ast g]/2=\eta_0 i_{\gm_\ast},$$ 
contradicting our assumption. 

To prove (2), consider the actions of $\G$ on $X_\ast=\widetilde{M}_{\gm}$ and $X=\widetilde{M_\ast}_{\gm_\ast}$, where $\G$ acts on $\widetilde{M_\ast}_{\gm_\ast}$ via $f_\ast $. By Lemma \ref{lem.bounddifgeometry} (2) we have that the action of $\G$ on $\widetilde{M_\ast}_{\gm_\ast}$ is $D_\ast$-cobounded with $D_\ast=C_n \|\G\|\lam^{-n}{i_{\gm_\ast}}^{1-n}=\hat{L}_\ast /6$.  The conclusion follows by Lemma \ref{lem.bounddifgeometry} (3) and Assumption \eqref{eq.ellRiemannian}, and Theorem \ref{thm.approxLSRvscobounded} applied to $X_\ast$, $X$ and $L$.
\end{proof}

The corollary above gives us one of the bounds in \eqref{eq.buttconclusion}. For the other bound, we would like to apply the corollary with the roles of $(M,\gm)$ and $(M_\ast,\gm_\ast)$ exchanged. To do this we need the following proposition.

\begin{proposition}\label{prop.bounddilfromdil}
    Let $(M,\gm), (M_\ast,\gm_\ast)$ be closed Riemannian manifolds of dimension $n\geq 2$, sectional curvatures contained in the interval $[-\Lam^2,-\lam^2]\subset \R_{<0}$, and injectivity radii bounded below by $i$. 
    Let $\G=\pi_1(M)$, and consider an isomorphism $f_\ast:\G \ra \pi_1(M_\ast)$. Then for all $P>1$ there exists a constant $R=R(\Gam,\lam,\Lam,i,P)>0$ such that if 
        $\ell_\gm[g]\leq P\ell_{\gm_\ast}[f_\ast g]$ for all $[g]\in \conj$, then
        \begin{equation}\label{eq.sup}
            \sup_{[g]\in \conj'}{\frac{\ell_{\gm_\ast}[f_\ast g]}{\ell_{\gm}[g]}} \leq R.
        \end{equation}
         Moreover, if $n\geq 3$ then $R$ does not depend on $P$. 
\end{proposition}
\begin{proof}
Let $\G$ act on $\widetilde{M}_{\gm}, \widetilde{M_\ast}_{\gm_\ast}$ as in the proof of Corollary \ref{coro.firstineq} (2). By Lemma \ref{lem.bounddifgeometry} and Remark \ref{rmk.RGvscobounded}, the metric structures $\rho_\gm=[\widetilde{M}_{\gm}]$ and $\rho_{\gm_\ast}=[\widetilde{M_\ast}_{\gm_\ast}]$ belong to $\scrD_\G^{\del,\al}$, for $\del$ and $\al$ depending only on $n,\|\G\|,\lam,\Lam$ and $i$. In addition, the function $\Psi:\scrD_\G\times \scrD_\G \ra \R$ given by
$\Psi([d_0],[d_1])\ra \Dil(d_0,d_1)$, where $d_0$ and $d_1$ have exponential growth rate 1, is continuous for the topology induced by $\Del$. This follows by combining Proposition 3.5 and Lemma 3.6 in \cite{reyes}. 

If $n\geq 3$, then $\Out(\G)$ is finite. Indeed, otherwise by \cite[Cor.~1.3]{bestvina.feighn.stable} $\G$ would split over a virtually cyclic subgroup. Since $\partial \G$ is homeomorphic to the $(n-1)$-dimensional sphere, this virtually cyclic subgroup would have to be infinite, but this contradicts \cite[Thm.~6.2]{bowditch} because $\partial \G$ has no local-cut points. Since $\G$ is torsion-free, by \cite[Thm.~1.7]{reyes} we deduce that $\scrD_\G^{\del,\al}$ is compact. The conclusion then follows by the continuity of $\Psi$, Lemma \ref{lem.bounddifgeometry} (4), and noting that the left-hand side of \eqref{eq.sup} equals $v_{\gm}\Psi(\rho_{\gm_\ast},\rho_\gm)/v_{\gm_\ast}$.

If $n=2$, then \cite[Thm.~1.7]{reyes} still implies that $\Out(\G)$ acts cocompactly on $\scrD_\G^{\del,\al}$, so let $\calB\subset \scrD_\G^{\del,\al}$ be a compact set such that $\Out(\G)\cdot \calB=\scrD_\G^{\del,\al}$. We further assume that $\rho_{\gm_\ast}\in \calB$, and that $D=\textnormal{Diam}_\Del(\calB)$ is minimal among all the diameters of compact subsets of $\scrD^{\del,\al}_\G$ whose translates by $\Out(\G)$ cover $\scrD_\G^{\del,\al}$. We also fix an arbitrary point $\rho_0=[d_0]\in \calB$ with exponential growth rate 1. By the same argument in the proof of \cite[Thm.~1.7]{reyes} we can prove that the function $\Out(\G) \ra \R$ that maps $\phi$ to $\Psi(\phi(\rho_0),\rho_0)=\Dil(\overline{\phi}(d_0),d_0)$ is proper, where $\ov{\phi}$ is any automorphism of $\G$ representing $\phi$. This implies that the set $\calF\subset \Out(\G)$ of outer automorphisms $\phi$ represented by an automorphism $\overline{\phi}$ such that $$\Dil(\overline{\phi}(d_0),d_0)\leq \hat{P}:=C_2\Lam \lam^{-1}e^{2D}P$$
is finite, where $C_2$ is the constant from Lemma \ref{lem.bounddifgeometry}. Therefore, if $\phi\in \Out(\G)$ is represented by $\overline{\phi}$ and $\rho\in \phi \calB$, then by Lemma \ref{lem.bounddifgeometry} (4) we have
\[
\Dil(\overline{\phi}(d_0),d_0)\leq \Dil(\overline{\phi}(d_0),\widetilde{M_\ast}_{\gm_\ast})\Dil(\widetilde{M_\ast}_{\gm_\ast},\widetilde{M}_{\gm})\Dil(\widetilde{M}_{\gm},d_0)\leq \hat{P},
\]
and hence $\phi\in \calF$. The conclusion follows by choosing 
\[
R=C_2\Lam \lam^{-1}\max\left\{\Psi(\rho_1,\rho_2)\colon \rho_1,\rho_2\in \bigcup_{\phi\in \calF}{\phi \calB}\right\}
\]
and noting that $D, \hat{P}, \calF$ and $R$ depend only on $\G$, $\del, \al$ and $P$. 
\end{proof}

\begin{proof}[Proof of Theorem \ref{thm.butt}]
    Let $n=1+\dim(\partial \G)$ be the dimension of $M$ and $M_\ast$, and let $i_{\gm_\ast}$ be the injectivity radius of $(M_\ast,\gm_\ast)$. Also, let $K,C_n$ be the constants from Theorem \ref{thm.approxLSRvscobounded} and Lemma \ref{lem.bounddifgeometry} respectively, and let $\hat{L}_\ast:=6C_n\|\G\|\lam^{-n}((1-\ep_0)i_\gm)^{1-n}$. By Corollary \ref{coro.firstineq}, if Assumption \eqref{eq.buttassumption} holds for $L> 2\max\{(1-\ep_0)i_\gm,\hat{L}_\ast\}$ then $\hat{L}_\ast \geq 6C_n\|\G\|\lam^{-n}i_{\gm_\ast}^{1-n}$, and hence for all $[g]\in \conj$ we have
    \begin{align*}
        \frac{\ell_{\gm}[g]}{\ell_{\gm_\ast}[f_\ast g]} &\leq (1-\ep)^{-1}\left(\frac{L}{L-\hat{L}_\ast}\right)+\frac{2K\lam^{-1}\log{2}}{L-\hat{L}_\ast} \\
        &\leq (1-\ep)^{-1}+\frac{2(\hat{L}_\ast(1-\ep_0)^{-1}+2K\lam^{-1}\log{2})}{L} \\
        &\leq (1-\ep_0)^{-1}+\frac{2(\hat{L}_\ast(1-\ep_0)^{-1}+2K\lam^{-1}\log{2})}{2\hat{L}_\ast}=:P. \\
    \end{align*}
In particular, for all $[g]\in \conj'$ we deduce
  \begin{align*}
        \frac{\ell_{\gm_\ast}[f_\ast g]}{\ell_{\gm}[g]} & 
        \geq \left( (1-\ep)^{-1}+\frac{2(\hat{L}_\ast(1-\ep_0)^{-1}+2K\lam^{-1}\log{2})}{L} \right)^{-1} \\
        & \geq (1-\ep)-\frac{2(\hat{L}_\ast(1-\ep_0)^{-1}+2K\lam^{-1}\log{2})}{L},\\
    \end{align*}
which settles the lower bound in \eqref{eq.buttconclusion}. 

For the upper bound, since $P$ only depends on $\G,\lam,\Lam,i_\gm$ and $\ep_0$, by Proposition \ref{prop.bounddilfromdil} there exists $R=R(\G,\lam,\Lam,i_\gm,\ep_0)>0$ such that $\ell_{\gm_\ast}[f_\ast g]\leq L$ holds for all $[g]\in \conj$ satisfying $\ell_{\gm}[x]\leq L/R$. In consequence, \eqref{eq.buttassumption} holds for all $[g]\in \conj'$ such that $\ell_{\gm}[g]\leq L/R$. If we define $L_\ast:=6C_n\|\G\|\lam^{-n}{i_{\gm}}^{-1+n}$, Corollary \ref{coro.firstineq} implies that for $L> 2RL_\ast$ and all $[g]\in \conj'$ we have
\begin{align*}
        \frac{\ell_{\gm_\ast}[f_\ast g]}{\ell_{\gm}[g]} & 
        \leq (1+\ep)\left(\frac{L}{L-RL_\ast}\right)+\frac{2RK\lam^{-1}\log{2}}{L-RL_\ast} \\
        &= (1+\ep)+\frac{R(L_\ast(1+\ep)+2K\lam^{-1}\log{2})}{L-RL_\ast}\\
        & \leq (1+\ep)+\frac{2R(L_\ast(1+\ep_0)+2K\lam^{-1}\log{2})}{L}.\\
    \end{align*}
This proves the upper bound in \eqref{eq.buttconclusion}, and the theorem follows by setting $L_0:=2\max \{RL_\ast,\hat{L}_\ast,(1-\ep_0)i_\gm\}+1$, which depends only on $\G,\lam,\Lam,i_\gm$ and $\ep_0$.
\end{proof}

\subsection*{Open access statement}
For the purpose of open access, the authors have applied a Creative Commons Attribution (CC BY) licence to any Author Accepted Manuscript version arising from this submission.


\noindent\small{Department of Mathematics, 
University of Warwick,
Coventry, CV4 7AL, UK}\\
\small{\textit{Email address}: \texttt{stephen.cantrell@warwick.ac.uk}\\
\\
\small{Department of Mathematics, Yale University, New Haven, CT 06511, USA}\\
\small{\textit{Email address}: \texttt{eduardo.c.reyes@yale.edu}}\\

\end{document}